\documentclass{amsart}
\usepackage{amsfonts, amsmath, amssymb, graphicx}
\newtheorem{definition}{\bf Definition}
\newtheorem{lemma}{\bf Lemma}
\newtheorem{theorem}{\bf Theorem}
\newtheorem{remark}{Remark}
\newtheorem{corollary}{\bf Corollary}
\newcommand{\N}{{\mathbb N}}
\begin{document}

\title[Feet in Buekenhout-Metz Unitals]{Feet in Orthogonal-Buekenhout-Metz Unitals}
\author{N. Abarz\'ua}
\address{Departamento de Matem\'aticas, Universidad Adolfo Iba\~nez}
\email{nicolas.abarzua@uai.cl}
\author{R. Pomareda}
\address{Departamento de Matem\'aticas, Universidad de Chile.}
\email{rpomared@uchile.cl}
\author{O. Vega}
\address{Department of Mathematics, California State University, Fresno.}
\email{ovega@csufresno.edu}
\subjclass{Primary 05, 51; Secondary 20}
\thanks{The authors would like to thank \emph{Universidad de Chile}, and its \emph{Stimulus Program for Institutional Excellence} for supporting the third author's visit to Universidad de Chile, where a part of this work was done.  Also, the second author was funded by Fondecyt project \# 1140510.}

\begin{abstract}
Given an Orthogonal-Buekenhout-Metz unital $U_{\alpha,\beta}$, embedded in $PG(2,q^2)$, and a point $P\notin U_{\alpha,\beta}$, we study the set of feet, $\tau_{P}(U_{\alpha,\beta})$, of $P$ in $U_{\alpha,\beta}$. We characterize geometrically each of these sets as either $q+1$ collinear points or as $q+1$ points partitioned into two arcs. Other results about the geometry of these sets are also given.
\end{abstract}

\maketitle

\section{Preliminares}

Most definitions and theorems in this section may be found in \cite{D68}. We direct the reader to this source for more information, and details, about projective planes.

Let $GF(q^{2})$ be the field with $q^{2}$ elements, where $q=p^{n}$ with $p$ prime (we will always consider $p$ to be odd in this article), and $n\in \N$. Throughout this article we will use
\[
GF(q^2) = \{ a+\epsilon b; \ a,b\in GF(q), \ \text{and} \ \epsilon^2=w\in GF(q) \},
\]
and we will write $\overline{x}=x^q$, $T(x) = x+ \overline{x}$ and $N(x)=x\overline{x}$, for all $x\in GF(q^2)$.

We let $V$ be a $3$-dimensional vector space over $GF(q^{2})$ and we consider the projective plane $\Pi = PG(2,q^{2})$, defined by letting its points to be  the $1$-dimensional subspaces of $V$ and its lines be the $2$-dimensional subspaces of $V$.  A point $P$ of $\Pi$ will be denoted by
\[
P=[a,b,c],
\]
where $(a,b,c)$ is a vector generating the subspace defining $P$.  If $l$ is a line of $\Pi$ then it will be denoted by
\[
l=\left[\begin{array}{c}
x\\
y\\
z\\
\end{array}\right]=[x,y,z]^{t},
\]
where $(x,y,z)$ is a vector that is orthogonal (using the standard dot product) to the $2$-dimensional subspace defining $l$.  The incidence in $\Pi$ is given by natural set-theoretic containment. Thus,
\[
P\in l \ \ \Longleftrightarrow \ \ 
\left[\begin{array}{ccc}
a & b & c\\
\end{array}\right]
\left[\begin{array}{c}
x\\
y\\
z\\
\end{array}\right] 
= 0 \ \ \Longleftrightarrow \ \ ax+by+cz=0.
\]

It is known that $\Pi$ is a projective plane of order $q^2$. Hence, the following properties hold in $\Pi$:
\begin{enumerate}
\item Every line of $\Pi$ contains exactly $q^{2}+1$ points.
\item Every point of $\Pi$ is on exactly  $q^{2}+1$ lines.
\item The number of points, and the number of lines, in $\Pi$ is $q^{4}+q^{2}+1$.
\end{enumerate}

In order to motivate the concept of unital, we define another important object.

\begin{definition}
A \emph{blocking set} $\beta$ is a subset of points of $\Pi$ such that every line of $\Pi$ contains at least one point in $\beta$.  A \emph{minimal blocking set} is a blocking set in which removing any of its points never yields a blocking set.
\end{definition}

\begin{remark}
The collection of points on a line $l$ of $\Pi$ is a blocking set.  We will say that a blocking set containing all the points on a line is a trivial blocking set. 
\end{remark}

A good summary of the basics on blocking sets may be found in Chapter 13 of  \cite{H79}. The following result may be found there.

\begin{theorem}\label{boundofblockingset}
Blocking sets exist in $\Pi$. Moreover, if $\beta$ is a minimal blocking set of $\Pi$ then $|\beta|\leq q^3 + 1$.
\end{theorem}

A very special kind of a largest possible minimal blocking set of $\Pi$ is the object we want to focus our attention from now on.

\begin{definition}
A \emph{unital} in $\Pi$ is a set $U$ of $q^{3}+1$ points of $\Pi$ such that every line of $\Pi$ intersects $U$ in exactly $1$ or $q+1$ points. Lines of $\Pi$ will be called \emph{tangent} or \emph{secant} to $U$ depending on whether they intersect $U$ in $1$ or $q+1$ points, respectively.
\end{definition}

\begin{remark}
Unitals may be defined in a much more general way but in this article we will focus only on unitals embedded in $\Pi$. So, our definition has been written with this purpose in mind. We refer the reader to \cite{BE08} for a detailed exposition about unitals and for the concepts we use in this article that we may fail to explain in detail.
\end{remark}

Two standard examples of unitals are  
\begin{enumerate}
\item The set of absolute points of a non-degenerate unitary polarity of $\Pi$:
\[
H=\{[x,y,z]\in \Pi; \  N(x)+N(y)+N(z)=0\}
\]
is a unital in $\Pi$, called \emph{classical}.
\item Buekenhout \cite{B76} proved that, for $\alpha, \beta \in GF(q^2)$ such that $4N(\alpha)+(\overline{\beta}-\beta)^{2}$ is a non-square in $GF(q)$,  the set
\[
U_{\alpha,\beta}=\{[x,\alpha x^{2}+\beta N(x)+r,1]; \ x\in GF(q^{2}), \ r\in GF(q)\}\cup\{P_{\infty}  \}
\]
is a unital (said to be an \emph{orthogonal-Buekenhout-Metz unital}) in $\Pi$, where $P_{\infty} =[0,1,0]$. Moreover, $\alpha=0$ if and only if the unital $U_{\alpha,\beta}$ is classical, and $\beta=\overline{\beta}$ if  and only if the unital $U_{\alpha,\beta}$ is a union of conics (see \cite{BE90} or \cite{HS91}, and \cite{DS13}).
\end{enumerate}

From now on we focus our study on non-classical orthogonal-Buekenhout-Metz unitals $U_{\alpha, \beta}$. So, for the rest of this article we assume $\alpha \neq 0$.

Elementary counting shows that if $U_{\alpha, \beta}$ is an orthogonal-Buekenhout-Metz unital in $\Pi$ and $P\in U_{\alpha, \beta}$ then there is exactly one tangent line to $U_{\alpha, \beta}$ through $P$ and there are exactly $q^{2}$ secant lines to $U_{\alpha, \beta}$ through $P$. Similarly, if $P\notin U_{\alpha, \beta}$ then there are exactly $q+1$ lines tangent to $U_{\alpha, \beta}$ through $P$  and there are exactly $q^{2}-q$ secant lines to $U_{\alpha, \beta}$ through $P$.

\begin{definition}
Let $U_{\alpha, \beta}$ be an orthogonal-Buekenhout-Metz unital and $P$ a point not in $U_{\alpha, \beta}$. Each of the $q+1$ points of $U_{\alpha, \beta}$ that are on a tangent line to $U_{\alpha, \beta}$ through $P$ is said to be a foot of $P$. We will denote the set of feet of $P$ by $\tau_{P}(U_{\alpha, \beta})$ and we will call it the \emph{pedal} of $P$.
\end{definition}

It is known that $\tau_{P}(U_{\alpha, \beta})$ has the following properties:
\begin{enumerate}
\item $\tau_{P}(U_{\alpha, \beta})$ is contained in a line of $\Pi$, for all $P\in \Pi\setminus U_{\alpha, \beta}$ if and only if $U_{\alpha, \beta}$ is classical (see Thas \cite{T92}). The conditions for this result have been relaxed after Thas's work, see \cite{AE02} for a more recent result on this characterization. 
\item For $U_{\alpha, \beta}$ non-classical. $\tau_{P}(U_{\alpha, \beta})$ is contained in a line of $\Pi$ if and only if $P\in \ell_{\infty}$. Note that $\ell_{\infty} \cap U_{\alpha, \beta} = P_{\infty}$ (see, e.g. \cite{BE08}). \\
Note that this result implies that every line, different from $\ell_{\infty}$, through $P_{\infty}$ contains a pedal.
\end{enumerate}

Finally, there is a group $G \leq P\Gamma L(3,q^{2})$ leaving $U_{\alpha,\beta}$ invariant and fixing $P_{\infty}$ such that 
\begin{enumerate}
\item $G$ is transitive on the set of points of $U_{\alpha,\beta} \setminus \ell_{\infty}$. 
\item $G$ is transitive on the points of $l_{\infty}\setminus \{P_{\infty}\}$.
\item $G$ has either one or two orbits on the points of $\Pi\setminus (U_{\alpha,\beta}\cup l_{\infty})$. Moreover, these orbits are those of $P_{\lambda}=[0,\lambda \epsilon,1]$, with $\lambda=1$ or $\lambda=w=\epsilon^{2}$.
\end{enumerate}

\section{Intersections of Lines and Pedals}

Our objective is to find geometric properties that can describe the pedals of points $P \notin  \ell_{\infty}$ in unitals $U_{\alpha,\beta}$, where $\alpha \neq 0$. In particular, we care about how lines of $\Pi$ intersect these sets. Not much is known about pedals in non-classical unitals, albeit the work by Kr{\v{c}}adinac and Smoljak  is pertinent; in \cite{KS11} they study all possible configurations for pedals in unitals that are embedded in (not-necessarily Desarguesian) projective planes of order $9$ and $16$.

Because of the results listed above about the group $G$ we will now only study $\tau_{P}(U_{\alpha,\beta})$ for $P = P_{\lambda}=[0,\lambda \epsilon,1]$, with $\lambda=1$ or $\lambda=w$; this decision is justified in the following lemma.

\begin{lemma}\label{lemorbits}
If $\sigma\in G$, $A,C\in \tau_{P}(U_{\alpha,\beta})$, $\sigma(A)=B$, $\sigma(C)=D$, $Q=\sigma(P)$, then $\sigma(\tau_{P}(U_{\alpha,\beta})) = \tau_{Q}(U_{\alpha,\beta})$ and
\[
| AC \cap \tau_{P}(U_{\alpha,\beta}) |  =  | BD \cap \tau_{Q}(U_{\alpha,\beta})|.
\]
\end{lemma}

\begin{proof}
It is easy to see that $\sigma$ preserves the number of points of intersection between lines and $U_{\alpha,\beta}$, and so $\sigma$ maps tangent lines into tangent lines. It follows that $\sigma(\tau_{P}(U_{\alpha,\beta})) = \tau_{Q}(U_{\alpha,\beta})$.

Note that $\sigma(AC)=BD$ and that $B,D \in \tau_{Q}(U_{\alpha,\beta})$. So, if we repeat this argument with $A$ and any other point $E \in AC \cap \tau_{P}(U_{\alpha,\beta})$ we would get another point in $BD \cap \tau_{Q}(U_{\alpha,\beta})$. Hence, since $\sigma$ is injective we get one direction of the desired inequality. We obtain the other direction by repeating the argument using $\sigma^{-1}$ instead of $\sigma$.
\end{proof}

We first look at the lines through $P_{\infty}$. It was mentioned earlier that there is a bijection between the set of pedals containing $P_{\infty}$ and the set of lines, different from $\ell_{\infty}$, through this point. We now want to look at how these $q^2$ lines intersect $\tau_{P_{\lambda}}(U_{\alpha,\beta})$. The following remark gives enough information for us to address this issue in the subsequent lemma. 

\begin{remark}
Two distinct pedals can intersect in at most one point, as for every point $A\in \tau_{P}(U_{\alpha,\beta})\cap \tau_{Q}(U_{\alpha,\beta})$ we always get that $A$, $P$ and $Q$ are collinear.  Moreover, two distinct pedals intersect if and only if they are the pedals of two points on a line tangent to $U_{\alpha,\beta}$; their intersection is the tangency point.
\end{remark}

The following lemma is immediate.

\begin{lemma}\label{lemlinesthroughPinfinity}
Let $\ell \neq \ell_{\infty}$ be a line such that $P_{\infty} \in \ell$. Then, $\ell$ is either tangent or exterior to all the pedals not contained in $\ell$.  
\end{lemma}

The generalization of this lemma to lines intersecting pedals of points not on $\ell_{\infty}$ is not true (see Section \ref{secnotPinfty}). However, Lemma \ref{lemlinesthroughPinfinity} implies that the line $\ell\neq \ell_{\infty}$ can be partitioned into singletons, all of them in distinct pedals. We are able to prove that result for all other pedals as well. 

\begin{lemma}
Let $\ell$ be a line that is not tangent to $U_{\alpha,\beta}$. Then, there are $q+1$ distinct pedals intersecting $\ell$ in singletons, creating a partition of the points in  $\ell \cap U_{\alpha,\beta}$.
\end{lemma}

\begin{proof}
The case when $P_{\infty} \in \ell$ follows immediately from Lemma \ref{lemlinesthroughPinfinity}. 

Now, if $P_{\infty} \notin \ell$ then we use that 
every point in $U_{\alpha,\beta}$ is in $q^2$ pedals, and $\ell$ contains $q+1$ points of $U_{\alpha,\beta}$ then for each point in $\ell \cap U_{\alpha,\beta}$ there are at least $q^2-(q+1)$ pedals containing no other point of $\ell \cap U_{\alpha,\beta}$. Hence, using that $q\geq 3$ implies that $q^2-(q+1)\geq q+1$, we can choose the pedals to create the desired partition.
\end{proof}

Now our interest shifts to learn about the intersections of lines, not through $P_{\infty}$, with pedals of points not on $\ell_{\infty}$.

\section{Lines Not Containing $P_{\infty}$}\label{secnotPinfty}

In this section we will study lines that do not go through $P_{\infty}$. We consider the orthogonal-Buekenhout-Metz unital in $\Pi$
\[
U_{\alpha,\beta}=\{[x,\alpha x^{2}+\beta N(x)+r,1]; \ x\in GF(q^{2}), \  r\in GF(q)\}\cup\{P_{\infty}\}.
\]

The tangent line to $U_{\alpha,\beta}$ through $[x,\alpha x^2+\beta N(x)+r,1]$ is
\[
[-2\alpha x+(\overline{\beta}-\beta)\overline{x},1,\alpha x^{2}-\overline{\beta}N(x)-r]^{t}.
\]

In order to study $\tau_{P_{\lambda}}(U_{\alpha,\beta})$ we need to determine all $x\in GF(q^{2})$ and $r\in GF(q)$ such that 
\[
[0,\lambda \epsilon,1]\in[-2\alpha x+(\overline{\beta}-\beta)\overline{x},1,\alpha x^{2}-\overline{\beta}N(x)-r]^{t}
\]
which means
\begin{align*}
[0,\lambda \epsilon,1]\left[\begin{array}{c}
-2\alpha x+(\overline{\beta}-\beta)\overline{x}\\
1\\
\alpha x^{2}-\overline{\beta}N(x)-r
\end{array}\right]
=0  & \Longleftrightarrow \ \lambda \epsilon+\alpha x^{2}-\overline{\beta}N(x)-r=0\\
&\Longleftrightarrow \ r=\lambda \epsilon +\alpha x^{2}-\overline{\beta}N(x).
\end{align*}

Since $r\in GF(q)$ we get $\overline{r}=r$. Hence,
\[
\overline{\lambda \epsilon +\alpha x^{2}-\overline{\beta}N(x)} \ =\lambda \epsilon +\alpha x^{2}-\overline{\beta}N(x)
\]
and thus
\begin{equation}\label{eqtogettrace}
2\lambda \epsilon+\alpha x^{2}-\overline{\alpha}\;\overline{x}^{2}+(\beta-\overline{\beta})N(x)=0.
\end{equation}

We let 
\[
M_{\alpha,\beta}=\left[\begin{array}{cc}
\alpha & \frac{1}{2}(\beta-\overline{\beta})\\
\frac{1}{2}(\beta-\overline{\beta}) & -\overline{\alpha}\\
\end{array}\right], 
\]
and notice that
\[
2\lambda \epsilon+\alpha x^{2}-\overline{\alpha} \ \overline{x}^{2}+(\beta-\overline{\beta})N(x)=0 \ \Longleftrightarrow  \ 2\lambda\epsilon+\left[\begin{array}{cc}
x & \overline{x}\\
\end{array}\right]M_{\alpha,\beta}\left[\begin{array}{c}
x\\
\overline{x}\\
\end{array}\right]=0.
\]

Hence,  $\tau_{P_{\lambda}}(U_{\alpha,\beta})$ is the set of all the points of the form 
\[
[x,2\alpha x^{2}+(\beta-\overline{\beta})N(x)+\lambda \epsilon,1], 
\]
where  $x\in GF(q^{2})$, and 
\[
2\lambda \epsilon+\left[\begin{array}{cc}
x & \overline{x}\\
\end{array}\right]M_{\alpha,\beta}\left[\begin{array}{c}
x\\
\overline{x}\\
\end{array}\right]=0.
\]

We can now use Equation (\ref{eqtogettrace}) to find a different way to represent points in $\tau_{P_{\lambda}}(U_{\alpha,\beta})$. Notice that
\begin{align*}
T(\alpha x^2) -\lambda \epsilon & = (\alpha x^2) + \overline{(\alpha x^2)}-\lambda \epsilon \\
& = \alpha x^2 + \overline{\alpha} \ \overline{x}^2-\lambda \epsilon \\
& = \alpha x^2 +  (    2\lambda \epsilon+\alpha x^{2}+(\beta-\overline{\beta})N(x)    ) -\lambda \epsilon \\ 
& = 2 \alpha x^2     +    \lambda \epsilon   +    (\beta-\overline{\beta})N(x).
\end{align*}

Hence, letting $Q_x=[x,T(\alpha x^2) -\lambda \epsilon,1]$ we get
\[
\tau_{P_{\lambda}}(U_{\alpha,\beta}) = \left\{Q_x; \ x\in GF(q^{2}), \ 2\lambda \epsilon+\left[\begin{array}{cc}
x & \overline{x}\\
\end{array}\right]M_{\alpha,\beta}\left[\begin{array}{c}
x\\
\overline{x}\\
\end{array}\right]=0\right\}.
\]

\begin{remark}\label{remimaginarynorm}
Equation (\ref{eqtogettrace}) can be re-written as
\[
2\lambda \epsilon+  2 \epsilon Im(\alpha x^{2})+(\beta-\overline{\beta})N(x)=0
\]

It follows that for two points $Q_x, Q_y \in \tau_{P_{\lambda}}(U_{\alpha,\beta})$ we get $N(x) =N(y)$ if and only if $Im(\alpha x^{2})= Im(\alpha y^{2})$.
\end{remark}

We now introduce some notation. Let 
\begin{align*}
T_{\lambda} & = \left\{ x\in GF(q^{2}); \ Q_x \in \tau_{P_{\lambda}}(U_{\alpha,\beta}) \right\} \\
& = \left\{x\in GF(q^{2}); \ 2\lambda \epsilon+\left[\begin{array}{cc}
x & \overline{x}\\
\end{array}\right]M_{\alpha,\beta}\left[\begin{array}{c}
x\\
\overline{x}\\
\end{array}\right]=0\right\}.
\end{align*}

It is easy to see that $x\in T_{\lambda}$ if and only if $-x\in T_{\lambda}$. Moreover, if $x, -x\in T_{\lambda}$ then they have the same value of $r$ associated to them (in the representation of $Q_x$ and $Q_{-x}$ as points in $U_{\alpha,\beta}$).

\begin{lemma}\label{lemmxand-xandythen-y}
Let $l_{x,-x}$ be the line through $Q_{x}$ and $Q_{-x}$, where $\pm x\in T_{\lambda}$.  If $Q_{y} \in l_{x,-x}$, for some $y\in T_{\lambda}$, then $Q_{-y} \in l_{x,-x}$.
\end{lemma}

\begin{proof}
The line passing through $Q_{x}$ and $Q_{-x}$ is given by
\[
l_{x,-x}=\left[\begin{array}{c}
0\\
-1\\
T(\alpha x^2) -\lambda \epsilon
\end{array}\right].
\]

If $Q_y \in l_{x,-x}$ then
\[
[y,T(\alpha y^2) -\lambda \epsilon,1]
\left[\begin{array}{c}
0\\
-1\\
T(\alpha x^2) -\lambda \epsilon
\end{array}\right] = 0,
\]
which can be simplified to
\[
T(\alpha x^2) -   T(\alpha y^2)  =0.
\]

On the other hand,
\[
[- y,T(\alpha (-y)^2) -\lambda \epsilon,1]
\left[\begin{array}{c}
0\\
-1\\
T(\alpha x^2) -\lambda \epsilon
\end{array}\right] = T(\alpha x^2) -   T(\alpha y^2),
\]
which is equal to zero. Hence, $Q_{-y} \in l_{x,-x}$.
\end{proof}

\begin{remark}\label{remlinesthrough100}
All lines of the form $l_{x,-x}$ pass through $[1,0,0]$.
\end{remark}

We want to learn about the conditions under which the line through $Q_{x}$ and $Q_{y}$, for $x, y \in T_{\lambda}$,  contains more points of $U_{\alpha,\beta}$ besides $Q_x$ and $Q_y$. 

\begin{lemma}\label{lemmsametrace}
Let $l_{x,y}$ be the line through $Q_{x}$ and $Q_{y}$, for $x\neq y$  in $T_{\lambda}$.  If $Q_z \in l_{x,y}$, $z \in T_{\lambda}$, and $T(\alpha x^2) = T(\alpha y^2)$, then $Q_{-z} \in l_{x,y}$.
\end{lemma}

\begin{proof}
The line  $l_{x,y}$ is given by:
\[
l_{x,y}=\left[\begin{array}{c}
T(\alpha x^{2}) - T(\alpha y^{2}) \\
y-x\\
xT(\alpha y^2) - yT(\alpha x^2)+(y-x)\lambda \epsilon
\end{array}\right].
\]

But, $T(\alpha x^2) = T(\alpha y^2)$ and $x\neq y$, so $l_{x,y}$ is represented by
\[
l_{x,y}=\left[\begin{array}{c}
0 \\
-1 \\
T(\alpha x^2) -\lambda \epsilon
\end{array}\right].
\]

If $Q_z \in l_{x,y}$ then, after routine simplifications, we get
\[
- T(\alpha z^2)  + T(\alpha x^2) =0.
\]

It follows that $l_{x,y}$ can be represented by
\[
l_{x,y}=\left[\begin{array}{c}
0 \\
-1 \\
T(\alpha z^2) -\lambda \epsilon
\end{array}\right],
\]
which is the line $l_{z,-z}$.
\end{proof}

\begin{theorem}\label{themonlytwo}
Let $U_{\alpha,\beta}$ be an orthogonal-Buekenhout-Metz unital with $\alpha\neq0$. Let $Q_x$ and $Q_y$ be two distinct points in $\tau_{P_{\lambda}}(U_{\alpha,\beta})$, and let $l_{x,y}$ be the line through them. If $T(\alpha x^{2})\neq T(\alpha y^{2})$ then
\[
l_{x,y}\cap \tau_{P_{\lambda}}(U_{\alpha,\beta})=\{Q_x, Q_y\}.
\]
\end{theorem}

\begin{proof} 
Suppose that $l_{x,y}$ contains a point $Q_z \in \tau_{P_{\lambda}}(U_{\alpha,\beta})$, different from $Q_x$ and $Q_y$, then there is a $\mu\in GF(q^{2})\setminus \{0\}$ such that
\[
Q_z = Q_x+\mu Q_y.
\]

Note that $1+\mu\neq0$, otherwise $Q_x+\mu Q_y\notin U_{\alpha,\beta}$. Then,
\[
[z,T(\alpha z^{2})-\lambda\epsilon,1] = Q_x+\mu Q_y = \left[\frac{x+\mu y}{1+\mu}, \frac{T(\alpha x^{2}) +\mu T(\alpha y^{2})}{1+\mu} -\lambda\epsilon,1\right].
\]

This expression implies
\[
T\left(\alpha\left(\frac{x+\mu y}{1+\mu}\right)^{2}\right) = \frac{T(\alpha x^{2}) +\mu T(\alpha y^{2})}{1+\mu},
\]
which we re-write as:
\[
(1+\mu)T\left(\alpha\left(\frac{x+\mu y}{1+\mu}\right)^{2}\right) = T(\alpha x^{2}) +\mu T(\alpha y^{2}).
\]

It follows that
\begin{equation}\label{eqtraces}
T\left(\alpha\left(\frac{x+\mu y}{1+\mu}\right)^{2}-\alpha x^{2}\right)=\mu T\left(\alpha y^{2}-\alpha\left(\frac{x+\mu y}{1+\mu}\right)^{2}\right).
\end{equation}

If
\[
T\left(\alpha\left(\displaystyle\frac{x+\mu y}{1+\mu}\right)^{2}-\alpha x^{2}\right)=T\left(\alpha y^{2}-\alpha\left(\displaystyle\frac{x+\mu y}{1+\mu}\right)^{2}\right)=0
\]
then
\[
T\left(\alpha\left(\frac{x+\mu y}{1+\mu}\right)^{2}\right)=T(\alpha x^{2}) \ \ \   \text{and} \ \ \   T(\alpha y^{2})=T\left(\alpha\left(\frac{x+\mu y}{1+\mu}\right)^{2}\right)
\]
and thus $T(\alpha x^{2})=T(\alpha y^{2})$, which contradicts our hypothesis. It follows that Equation (\ref{eqtraces}) implies $\mu\in GF(q)$, and thus $1+\mu \in GF(q)$.

Since $z \in T_{\lambda}$ and $\displaystyle{z = \frac{x+\mu y}{1+\mu}}$ we get
\[
2\lambda \epsilon+\alpha \left(\displaystyle{\frac{x+\mu y}{1+\mu}}\right)^{2}-\overline{\alpha}\overline{\left(\displaystyle{\frac{x+\mu y}{1+\mu}}\right)}^{2}+(\beta-\overline{\beta}) \left(\displaystyle{\frac{x+\mu y}{1+\mu}}\right)^{q+1} = 0,
\]
which is equivalent to
\[
2\lambda\epsilon(1+\mu)^{2}+\alpha(x+\mu y)^{2}-\overline{\alpha}\overline{(x+\mu y)}^{2}+(\beta-\overline{\beta})(x+\mu y)^{q+1}=0.
\]

After some simplifications we get
\begin{align}\label{eqnalmostthere}
\begin{split}
\left(2\lambda\epsilon +  \alpha x^2  -  \overline{\alpha} \ \overline{x}^2   +  (\beta-\overline{\beta})N(x)\right) \\
+ \mu^2 \left(2\lambda\epsilon + \alpha y^2 -   \overline{\alpha}\  \overline{y}^{2}+ (\beta-\overline{\beta}) N(y)\right)     \\
  + \mu \left( 4\lambda\epsilon   + 2 \alpha  xy  -  2 \overline{\alpha} \ \overline{x} \ \overline{y} + (\beta-\overline{\beta})(\overline{y}x + \overline{x}y) \right)  & =  0.
\end{split}
\end{align}

Since $x, y  \in T_{\lambda}$, we know
\[          
2\lambda \epsilon+\alpha x^{2}-\overline{\alpha}\;\overline{x}^{2}+(\beta-\overline{\beta})N(x) = 
2\lambda \epsilon+\alpha y^{2}-\overline{\alpha}\;\overline{y}^{2}+(\beta-\overline{\beta})N(y) = 0
\]
and thus, Equation (\ref{eqnalmostthere}) implies
\[
4\lambda\epsilon   + 2 \alpha  xy  -  2 \overline{\alpha} \ \overline{x} \ \overline{y} + (\beta-\overline{\beta})(\overline{y}x + \overline{x}y)   =0,
\]
as $\mu\neq0$.

Since the expression above is zero, independent of the value of $\mu$, every point  $Q_x+\mu Q_y$, for $\mu\in GF(q)\setminus \{-1\}$, is in $\tau_{P_{\lambda}}(U_{\alpha,\beta})$. Hence, there are exactly $q$ feet of $P_{\lambda}$ in $U_{\alpha,\beta}$ lying on the same line $\ell$. But since $x\in T_{\lambda}$ implies $- x\in T_{\lambda}$, there is a $y\in T_{\lambda}$ such that $\ell = l_{y,-y}$. However, by Lemma \ref{lemmxand-xandythen-y} the number of points on $\ell \cap \tau_{P_{\lambda}}(U_{\alpha,\beta})$ must be even, which contradicts $q$ being odd.
\end{proof}

As of now we have that a secant line cannot intersect $\tau_{P_{\lambda}}(U_{\alpha,\beta})$ in $3$ points, and that if the line is not of the form $l_{x,-x}$ then this intersection contains at most $2$ points. Next we obtain a bound for the maximum number of collinear points on  $\tau_{P_{\lambda}}(U_{\alpha,\beta})$.

\begin{theorem}\label{thmatmostfour}
Let $U_{\alpha,\beta}$ be an orthogonal-Buekenhout-Metz unital with $\alpha\neq0$ and let $\ell$ be the line through two distinct points in $\tau_{P_{\lambda}}(U_{\alpha,\beta})$. Then, $\ell$ intersects $\tau_{P_{\lambda}}(U_{\alpha,\beta})$ in at most four points.
\end{theorem}

\begin{proof}
Because of the previous results, the only case to consider is when $\ell$ is of the form $l_{x,-x}$, for some $x\in T_{\lambda}$. Hence, the conditions for $Q_z \in \ell \cap \tau_{P_{\lambda}}(U_{\alpha,\beta})$ are
\begin{equation}\label{eqnewtowardsBezout}
T(\alpha x^2) -   T(\alpha z^2)  =0 \hspace{.3in} \text{and} \hspace{.3in} 2\lambda \epsilon+  2 \epsilon Im(\alpha z^{2})+(\beta-\overline{\beta})N(z)=0.
\end{equation}

We let $z=z_{1}+z_{2}\epsilon$, $\alpha=\alpha_{1}+\alpha_{2}\epsilon$, and $\beta=\beta_{1}+\beta_{2}\epsilon$, where $z_{1}$, $z_{2}$, $\alpha_{1}$, $\alpha_{2}$, $\beta_1$, $\beta_2 \in GF(q)$.  Using these variables we can re-write Equations (\ref{eqnewtowardsBezout}) as the system
\begin{align}\label{eqnpointontangentline4}
\begin{split}
2^{-1} T(\alpha x^2)  & = \alpha_{1} z_{1}^2 +  \alpha_{1}  w z_{2}^2  + 2\alpha_{2}w z_1z_2              \\
- \lambda & = (\alpha_{2} +  \beta_2) z_{1}^2  + (\alpha_{2} - w \beta_2)  z_2^2  + 2\alpha_{1}z_1z_{2}  
\end{split}
\end{align}

We define the following elements in $GF(q)$.
\[
\begin{array}{llll}
A = \alpha_1  & B =  \alpha_{1}w & C =  2\alpha_{2} w \ \ \ \ \ & D =  - 2^{-1} T(\alpha x^2)   \\
E  =  \alpha_{2} +  \beta_2 \ \ \ \ \ & F =  \alpha_{2} - w \beta_2 \  \ \ \ \ & G =  2\alpha_{1} & H =   \lambda.
\end{array}
\]

These elements allow us to re-write System (\ref{eqnpointontangentline4}) as the following system of equations with coefficients in $GF(q)$:
\begin{equation}\label{eqnpointontangentline5}
A z_{1}^2 + B z_{2}^2  + C z_1z_2  +   D  =0        \hspace{.4in} \text{and} \hspace{.4in}  E z_{1}^2  + F z_2^2  + G z_1z_{2}  + H =0
\end{equation}

If these equations have a common linear factor then we get three linear equations equal to zero, which is three intersecting lines. This yields one solution or a triplet of coinciding lines, which would imply that each equation in System (\ref{eqnpointontangentline5}) is a multiple of the other. But we know that the equation $T(\alpha x^2) -   T(\alpha z^2)  =0$ has exactly $2(q+1)$ solutions, implying that System (\ref{eqnpointontangentline5}) has $2(q+1)$ solutions, which is more than the maximum number of points on $\ell \cap \tau_{P_{\lambda}}(U_{\alpha,\beta})$, which is $q+1$.

In the case the equations in System (\ref{eqnpointontangentline5}) do not have common factors we can use B\'ezout's Theorem for the curves given by
\begin{align*}
p(z_{1},z_{2}) &= Az_1^2 +  B z_{2}^2 +Cz_1z_{2} + D \\
q(z_{1},z_{2}) &= Ez_1^2 + F z_{2}^2 +Gz_1z_{2} + H 
\end{align*}
and since both are polynomials in two variables with coefficients in $GF(q)$, and both have degree two, we get that System (\ref{eqnpointontangentline4}) has at most $4=deg(p)\cdot deg(q)$ solutions.
\end{proof}

We summarize our results on the size of the intersections between lines and pedals in the following theorem.

\begin{theorem}
Let $P\notin \ell_{\infty}$ and $\alpha \neq 0$. Then,
\begin{enumerate}
\item lines in $\Pi$ intersect $\tau_{P}(U_{\alpha,\beta})$ in exactly $0$, $1$, $2$, or $4$ points. 
\item the points of $\tau_{P}(U_{\alpha,\beta})$ may be partitioned into two arcs. 
\end{enumerate}
\end{theorem}

\begin{proof}
The first part of the theorem follows from Theorems \ref{themonlytwo} and  \ref{thmatmostfour}, and Lemmas \ref{lemorbits}, \ref{lemmxand-xandythen-y} and \ref{lemmsametrace}.

For the second part, we use Lemma \ref{lemorbits} to allow ourselves to consider the particular case $P=P_{\lambda}$. Since we know that only the lines of the form $l_{x,-x}$ can intersect $\tau_{P_{\lambda}}(U_{\alpha,\beta})$ in four points, we look at these lines  first. 

Assume that the lines of the form $l_{x,-x}$ intersecting $\tau_{P_{\lambda}}(U_{\alpha,\beta})$ are partitioned as follows: $\ell_1, \ell_2, \ldots, \ell_n$ intersect $\tau_{P_{\lambda}}(U_{\alpha,\beta})$ in exactly two points and $\ell_{n+1}, \ell_{n+2}, \ldots, \ell_t$ intersect $\tau_{P_{\lambda}}(U_{\alpha,\beta})$ in four. We label the points of $\tau_{P_{\lambda}}(U_{\alpha,\beta})$ by $Q_{x}$, where $x$ is one of the following
\[
x_{11}, x_{12},\ldots, x_{n1}, x_{n2}, x_{(n+1)1}, x_{(n+1)2}, x_{(n+1)3}, x_{(n+1)4}, \ldots, x_{t1}, x_{t2}, x_{t3}, x_{t4}
\]
where $Q_{x_{ij}}\in \ell_i$, and $x_{i2}=-x_{i1}$ and $x_{i3}=-x_{i4}$, for all $i$.
Then, the points of $\tau_{P_{\lambda}}(U_{\alpha,\beta})$ can be partitioned into the following two arcs
\[
A_1 = \{Q_{x}; \ x= x_{ij}, \ i=1,\ldots,t \ \text{and} \ j=1,2 \} \cup \{P_{\lambda}\}
\]
and
\[
A_2  =   \{Q_{x};  \  x= x_{ij}, \  i=n+1,\ldots,t \ \text{and} \ j=3,4 \}.
\]

Note that the partition given is just one of the many possible ones.
\end{proof}

In the particular case when $\beta = \overline{\beta}$ we can get an even stronger result.

\begin{corollary}\label{corarcsbeta}
If $\alpha \neq 0$ and $\beta = \overline{\beta}$, then the points of $\tau_{P}(U_{\alpha,\beta})$ are contained in lines or arcs.
\end{corollary}

\begin{proof}
We already know that $\tau_{P}(U_{\alpha,\beta})$ is contained in a line when $P\in \ell_{\infty}$. For when $P\notin \ell_{\infty}$ we use Lemma \ref{lemorbits} to restrict ourselves to study the structure of $\tau_{P_{\lambda}}(U_{\alpha,\beta})$.

Let $x\neq y$ and let $l_{x,y}$ be the line through $Q_x, Q_y \in \tau_{P_{\lambda}}(U_{\alpha,\beta})$. Since $\beta=\overline{\beta}$ we get that $T(\alpha x^2) = 2 \alpha x^2 + 2 \lambda \epsilon$, for all $x\in T_{\lambda}$ (this follows from the argument before Remark \ref{remimaginarynorm}). Hence, a point $Q_z \in \tau_{P_{\lambda}}(U_{\alpha,\beta})$  now looks like $Q_z=[z,2\alpha z^{2}+\lambda \epsilon,1]$, and the line $l_{x,y}$ is given by
\[
\left[\begin{array}{c}
- 2\alpha(x+y)\\
1\\
2\alpha xy-\lambda \epsilon
\end{array}\right].
\]

Thus, $Q_z \in l_{x,y}$ if and only if 
\[
2\alpha[(x^{2}-y^{2})z-(x-y)z^{2}-xy(x-y)]=0.
\]

Since $\alpha\neq0$ and $x\neq y$, this equation reduces to 
\[
z^{2}-(x+y)z+xy=0,
\]
which can be re-written as
\[
(z-x)(z-y)=0.
\]

The result follows.
\end{proof}

\section{The Elation Group of $U_{\alpha,\beta}$}

Let us consider the collineation group of $U_{\alpha,\beta}$ given by 
\[
\mathcal{E}= \left\{ E_t: (x,y,z) \mapsto (x,y+tz,z); \ t\in GF(q) \right\}. 
\]

Note that $\mathcal{E}$ is an elation group with center $P_{\infty}$ and axis  $\ell_{\infty}$. It is easy to show that lines of the form $AE_t(A)$ must pass through $P_{\infty}$, for all $A\notin \ell_{\infty}$ and $E_t\in \mathcal{E}$. Also, since $\mathcal{E}$ acts semi-regularly on points not on its axis, nothing but the identity in $\mathcal{E}$ stabilizes a $\tau_{P}(U_{\alpha,\beta})$, and if $Q= E_t(P)$, for some $E_t \in \mathcal{E}$, we get that  $ \tau_{P}(U_{\alpha,\beta})$ and $\tau_{Q}(U_{\alpha,\beta})$ are disjoint. 

It has been mentioned before that every line through $P_{\infty}$, except from $\ell_{\infty}$, contains a pedal (of a point on  $\ell_{\infty}$). Moreover, it is easy to see that the $q$ points, different from  $P_{\infty}$, on each of these pedals form an orbit under the group $\mathcal{E}$. We take this observation as a `suggestion' to take a closer look at the orbits of pedals under $\mathcal{E}$ and to study how lines intersect these sets. 

From now on, we will use $\mathcal{O}(X)$ to denote the orbit of a set $X$ under the group $\mathcal{E}$. 

\begin{lemma}\label{lemsomelinesinq+1}
Given a pedal  $\tau_{P}(U_{\alpha,\beta})$, there is a point $Q\in \ell_{\infty}$ and $q$ lines through $Q$ that partition $\mathcal{O}(\tau_{P}(U_{\alpha,\beta}))$. That is, the intersection of each of these lines with $U_{\alpha,\beta}$ is completely contained in $\mathcal{O}(\tau_{P}(U_{\alpha,\beta}))$.
\end{lemma}

\begin{proof}
Because of Lemma \ref{lemorbits}, it is enough to look at how lines intersect $\tau_{P_{\lambda}}(U_{\alpha,\beta})$. We consider the point $[1,0,0]$ and the lines through it. We know that lines, different from $\ell_{\infty}$, through $[1,0,0]$ look like 
\[
l_{\gamma}=\left[\begin{array}{c}
0 \\
-1 \\
\gamma
\end{array}\right].
\]

It is easy to see that the orbit of $P_{\lambda}$ is contained on a line through $P_{\infty}$. So, we let $\mathcal{O}(P_{\lambda})=\{P_1, P_2, \ldots, P_q \}$, where $P_t = E_t(P_{\lambda})$, for all $t\in GF(q)$. Moreover, since 
\[
E_t(x,y,z)=(x,y+tz,z)
\]
for all $t\in GF(q)$, and  $P_{\lambda}=[0,\lambda \epsilon,1]$, with $\lambda=1$ or $\lambda=w$, we obtain
\[
P_t=[0,\lambda \epsilon+t,1]
\]

Using the arguments at the beginning of Section \ref{secnotPinfty}, we get that
\[
\tau_{P_{t}}(U_{\alpha,\beta}) = \left\{R_y; \ y\in GF(q^{2}), \ 2\lambda \epsilon+\left[\begin{array}{cc}
y & \overline{y}\\
\end{array}\right]M_{\alpha,\beta}\left[\begin{array}{c}
y\\
\overline{y}\\
\end{array}\right]=0\right\}
\]
where $R_y=[y,T(\alpha y^2) -\lambda \epsilon +t,1]$.

Now,  the points of intersection (if any) of $\tau_{P_{t}}(U_{\alpha,\beta})$ with $l_{\gamma}$ are given by
\[
0 = 
[y,T(\alpha y^2) -\lambda \epsilon +t,1]\left[\begin{array}{c}
0 \\
-1 \\
\gamma
\end{array}\right]  =  
-(T(\alpha y^2) -\lambda \epsilon +t)+ \gamma 
\]
which means 
\begin{equation}\label{eqalmostdone}
t= \gamma + \lambda \epsilon - T(\alpha y^2).
\end{equation} 

It follows that, if $y$ and $\gamma$ were given, and  $\gamma  = s - \lambda \epsilon$ for some $s \in GF(q)$, then we can always find a $t\in GF(q)$ that satisfies Equation (\ref{eqalmostdone}).  In this case, given a line $l_{s - \lambda \epsilon}$, for every $y\in GF(q^2)$ such that $R_y \in \tau_{P_{t}}(U_{\alpha,\beta})$ there is a point of intersection between $l_{s - \lambda \epsilon}$ and $\mathcal{O}(\tau_{P_{\lambda}}(U_{\alpha,\beta}))$.

Hence, for $\gamma \neq s - \lambda \epsilon$, for all $s\in GF(q)$ the intersection is empty and for when $\gamma = ks - \lambda \epsilon$ then the intersection contains $q+1$ points. Note that for every $s\in GF(q)$ we are able to choose such a $\gamma$, thus we get $q$ lines through $[1,0,0]$ intersecting $\mathcal{O}(\tau_{P_{\lambda}}(U_{\alpha,\beta}))$ in $q+1$ points each.
\end{proof}

We would like to close this paper stating a few open problems.
\begin{enumerate}
\item Do lines intersecting pedals in at least four points exist if and only if $\beta \neq \overline{\beta}$? Corollary \ref{corarcsbeta} gives us one direction of this conjecture.
\item What geometric properties determine when a line of the form $l_{x,-x}$ intersects a given $\tau_{P}(U_{\alpha,\beta})$ in four points? 
\item When a $\tau_{P}(U_{\alpha,\beta})$ is partitioned into two arcs (or contained in one arc for the case $\beta = \overline{\beta}$), is any of these arcs contained in a conic?
\item In how many points does a line intersect the set of points in the orbit of any given $\tau_{P}(U_{\alpha,\beta})$ under $\mathcal{E}$? Lemma \ref{lemsomelinesinq+1} gives us a partial answer to this, but there are several other lines that are not considered in this result.
\item Is there a combinatorial characterization for the structure formed by the lines of $\Pi$ and the points on $\mathcal{O}(\tau_{P_{\lambda}}(U_{\alpha,\beta}))$?
\end{enumerate}


\end{document}